\documentclass[12pt]{amsart}
\usepackage{amssymb, amscd, amsmath, amsthm, epsf, epsfig, latexsym, enumerate}

\newtheorem{theorem}{Theorem}
\newtheorem{lemma}[theorem]{Lemma}

\begin{document}

\title{The $\mathbb{F}_2$-cohomology rings of $\mathbb{S}ol^3$-manifolds}
 
\author{J.A.Hillman}
\address{School of Mathematics and Statistics\\
     University of Sydney, NSW 2006\\
      Australia }

\email{jonathan.hillman@sydney.edu.au}

\begin{abstract}
We compute the rings $H^*(N;\mathbb{F}_2)$ for $N$ a 
closed $\mathbb{S}ol^3$-manifold, 
and then determine the Borsuk-Ulam indices $BU(N,\phi)$ with 
$\phi\not=0$ in $H^1(N;\mathbb{F}_2)$.
\end{abstract}



\maketitle

The Borsuk-Ulam Theorem states that any continuous function 
$f:S^n\to\mathbb{R}^n$ takes the same value at some antipodal pair of points.
This may be put in a broader context as follows.
Let $N$ be an $n$-manifold and let $N_\phi$ be the double cover 
associated to an epimorphism $\phi:\pi\to{Z/2Z}$.
Let $t_\phi$ be the covering involution.
The {\it Borsuk-Ulam index\/} $BU(N,\phi)$ is the maximal value of $k$ 
such that for all maps $f:N_\phi\to\mathbb{R}^k$ 
there is an $x\in{N_\phi}$ with $f(x)=f(t_\phi(x))$.
Then the Borsuk-Ulam Theorem is equivalent to the assertion that $BU(RP^n,\alpha)=n$, 
where $\alpha:\pi_1(RP^n)\to{Z/2Z}$ is the canonical epimorphism.

In low dimensions this invariant may be determined cohomologically,
and is known for many pairs $(N,\phi)$, with $N$ a Seifert fibred 3-manifold, 
including all those with geometry $\mathbb{E}^3$, $\mathbb{S}^3$,
$\mathbb{S}^2\times\mathbb{E}^1$, $\mathbb{N}il^3$
or $\mathbb{H}^2\times\mathbb{E}^1$ \cite{GHZ,BGHZ}.
Here we shall determine this invariant for all such pairs
with $N$ a closed $\mathbb{S}ol^3$-manifold.
This follows easily once we know
the mod-2 cohomology rings of such manifolds.
We compute these using Poincar\'e duality and elementary 
properties of cup-product in the low-degree cohomology of groups.
(Our approach can also be applied to $\mathbb{E}^3$- and
$\mathbb{N}il^3$-manifolds.)

I would like to thank the organizers of the XVIII Encontro Brasiliero de Topologia for the invitation to their meeting in Aguas de Lindoias, 
in July 2012, which lead to this work,
and S.T.Martins, for sending me a copy of his PhD thesis \cite{Ma12}.
The use of the extraspecial 2-group $E$ (introduced before Lemma 6 below)
was prompted by the work of J.F.Carlson on the cohomology rings of 2-groups
\cite{jfc}.

\section{$\mathbb{S}ol^3$-manifolds and their groups}

Let $M$ be a closed $\mathbb{S}ol^3$-manifold.
Then $\pi=\pi_1(M)$ has an unique maximal abelian normal subgroup
$\sqrt\pi$,
which is free abelian of rank 2.
(This subgroup is in fact the Hirsch-Plotkin radical \cite{Ro} of $\pi$.)
The quotient $\pi/\sqrt\pi$ is virtually $\mathbb{Z}$
(i.e., has two ends), and so is an extension of
$\mathbb{Z}$ or $D_\infty=Z/2Z*Z/2Z$ by a finite normal subgroup.
The preimage of this finite normal subgroup is torsion-free, 
and so is either $\mathbb{Z}^2$ or  $\mathbb{Z}\rtimes_{-1}\mathbb{Z}$
(the Klein bottle group).
Since $Out(\mathbb{Z}\rtimes_{-1}\mathbb{Z})$ is finite
and $\pi$ is not virtually abelian,
this preimage must be $\sqrt\pi$.
Hence $\pi/\sqrt\pi\cong\mathbb{Z}$ or $D_\infty$.

Suppose first that $\pi/\sqrt\pi\cong\mathbb{Z}$.
Then $M$ is the mapping torus of a self-homeomorphism of $T=S^1\times{S^1}$,
and $\pi\cong\mathbb{Z}^2\rtimes_\Theta\mathbb{Z}$,
where $\Theta=\left(\begin{smallmatrix}
a&c\\ 
b&d
\end{smallmatrix}\right)\in{GL}(2,\mathbb{Z})$.
Thus $\pi$ has a presentation
\[
\langle{t,x,y}\mid{txt^{-1}=x^ay^b},~tyt^{-1}=x^cy^d,~xy=yx\rangle.
\]
Let $\varepsilon=\det\Theta=\pm1$ and $\tau=tr\Theta=a+d$.
Then $M$ is orientable if and only if $\varepsilon=1$, 
in which case $|\tau|>2$, 
since $\pi$ is not virtually nilpotent.
Let $\theta$ be a root of  ${\det(\Theta-XI_2)}=X^2-\tau{X}+\varepsilon$,
the characteristic polynomial of $\Theta$.
Then $\theta$ is a unit in the quadratic number field $\mathbb{Q}[\theta]$, 
and $\sqrt\pi$ is isomorphic to an ideal $I$ in the ring $\mathbb{Z}[\theta]$.
(The latter may not be the full ring of integers in $\mathbb{Q}[\theta]$!)

There is a converse.
Let $[I]$ denote the isomorphism class of the ideal $I$.
The Galois involution of the quadratic field $\mathbb{Q}[\theta]$
acts on the ring $\mathbb{Z}[\theta]$,
since $\bar\theta=\tau-\theta\in\mathbb{Z}[\theta]$,
and hence acts on the set of ideal classes.

\begin{theorem}
Let $\alpha$ be a quadratic algebraic unit which is not a root of unity, 
and let $J$ be a nonzero ideal in $\mathbb{Z}[\alpha]$.
Let $A$ be the automorphism of $J\cong\mathbb{Z}^2$ 
given by left multiplication by $\alpha$, 
and let $\pi=J\rtimes_A\mathbb{Z}$.
Then 

\begin{enumerate}
\item $\pi$ is a $\mathbb{S}ol^3$-group;

\item the groups corresponding to two such pairs $(\alpha,J)$ 
and $(\beta,K)$ are isomorphic if and only if either
$\beta=\alpha$ or $\alpha^{-1}$ and $[K]=[J]$,
or $\beta=\bar\alpha$ or $\bar\alpha^{-1}$ 
and $[K]=\overline{[J]}$;

\item  
given $\alpha$, the number of isomorphism classes of such groups $\pi$ 
 is finite.
\end{enumerate}
\end{theorem}

\begin{proof}
The group $\pi$ is the fundamental group of the mapping torus of 
a self-homeomorphism of $T$.
If $\alpha$ is not a root of unity then this is a $\mathbb{S}ol^3$-manifold.

Let $\pi=\langle{J,t}|tjt^{-1}=\alpha{j}~\forall{j}\in{J}\rangle$ 
and $\widetilde\pi=\langle{k,\widetilde{t}}|
\widetilde{t}j\widetilde{t}^{-1}=\beta{k}~\forall{k}\in{K}\rangle$ 
be two such groups.
An isomorphism $f:\pi\cong\widetilde\pi$ restricts to an isomorphism
$f_J:J=\sqrt\pi\cong\sqrt{\widetilde\pi}=K$.
Hence it induces an isomorphism 
$\pi/\sqrt\pi\cong\widetilde\pi/\sqrt{\widetilde\pi}$,
and so $f(t)=\widetilde{t}^\eta{k}$, for some $\eta=\pm1$ and $k\in{K}$.
We may assume that $f(t)=\widetilde{t}$, 
after replacing $\beta$ by $\beta^{-1}$, if necessary.
The characteristic polynomials of the automorphism of $J$ and $K$ 
induced by conjugation by $t$ and $\widetilde{t}$ (respectively) must 
then agree.
Thus either $\beta=\alpha$ and $f_J$ is 
an isomorphism of $\mathbb{Z}[\alpha]$-modules,
or $\beta=\bar\alpha$ and $\overline{f_J}:J\to\overline{K}$ is an 
an isomorphism of $\mathbb{Z}[\alpha]$-modules.
The converse is similarly straightforward.

The group $\pi$ is determined up to a finite ambiguity by 
$\alpha$ (equivalently, by the polynomial $t^2-\tau{t}+\varepsilon$),
since $\mathbb{Z}[\alpha]$ has finitely many ideal classes,
by the Jordan-Zassenhaus Theorem.
\end{proof}

If $\pi/\sqrt\pi\cong{D_\infty}$ then $\pi\cong{B*_TC}$, 
where $B$ and $C$ are torsion-free, $T\cong\mathbb{Z}^2$ and 
$[B:T]=[C:T]=2$.
Thus $M$ is the union of two twisted $I$-bundles.
Since $\beta_1(\pi;\mathbb{Q})=0$ and $\chi(M)=0$, $M$ is orientable,
and so $B$ and $C$ must be copies of the Klein bottle group.
Hence $M$ is the union of two copies of the mapping cylinder 
of the double cover of the Klein bottle.
The double cover of $M$ corresponding to the preimage 
of $\sqrt{D_\infty}$ in $\pi$ is a mapping torus.

In particular, $\pi$ has a presentation
\[\langle
u,v,y,z\mid uyu^{-1}=y^{-1}\!,~vzv^{-1}=z^{-1}\!,~yz=zy,~v^2=u^{2a}y^b,~
z=u^{2c}y^d
\rangle,
\]
where
$\left(\begin{smallmatrix}
a&c\\ 
b&d
\end{smallmatrix}\right)\in{GL}(2,\mathbb{Z})$ 
corresponds to the identification of $\sqrt{C}$ with $T=\sqrt{B}$.
This presentation simplifies immediately to
\[
\langle
u,v,y\mid uyu^{-1}=y^{-1},~v^2=u^{2a}y^b,~
vu^{2c}y^dv^{-1}=u^{-2c}y^{-d}
\rangle.
\]
Hence $\pi^{ab}\cong{Z/4cZ}\oplus{Z/4Z}$ if $b$ is odd,
and $\pi^{ab}\cong{Z/4cZ}\oplus(Z/2Z)^2$ if $b$ is even.
Let $x=u^2$.
Then conjugation by $uv$ acts on $\langle{x,y}\rangle\cong\mathbb{Z}^2$
via 
$\Psi=\eta\left(\begin{smallmatrix}
ad+bc&2ac\\ 
2bd&ad+bc
\end{smallmatrix}\right)$, where $\eta=ad-bc=\pm1$.
Hence $\det\Psi=1$, $\Psi\equiv{I_2}$ {\it mod} $(2)$
and $tr\Psi\equiv2$ {\it mod} $(4)$.
(These conditions are not independent,
for if $\Psi=I_2+2N$ then $tr\Psi=2+2trN$ and
$\det\Psi\equiv1+2trN$ {\it mod} (4), 
so $trN$ is even and $tr\Psi\equiv2$ {\it mod} $(4)$ 
if also $\det\Psi=1$.)
Moreover, $abcd\not=0$, since $M$ is not flat.

Conversely, any  $\left(\begin{smallmatrix}
a&c\\ 
b&d
\end{smallmatrix}\right)\in{GL}(2,\mathbb{Z})$ with $abcd\not=0$
gives rise to such a $\mathbb{S}ol^3$-manifold,
for then $|tr\Psi|=2|ad+bc|\geq6$.
Moreover,
suppose 
$P=\left(\begin{smallmatrix}
2k+1&2m\\ 
2n&2k+1
\end{smallmatrix}\right)\in{SL(2,\mathbb{Z})}$,
where $mn\not=0$.
Then $k(k+1)=mn$, and so we may write $m=m_1m_2$ and $n=n_1n_2$, 
with $k=m_1n_1$ and $k+1=m_2n_2$.
The $\mathbb{S}ol^3$-rational homology sphere 
corresponding to  $\left(\begin{smallmatrix}
m_1&-m_2\\ 
-n_2&n_1
\end{smallmatrix}\right)\in{GL}(2,\mathbb{Z})$, 
is doubly covered by the mapping torus asociated to $P$.

The above matrix calculations show that
a quadratic unit $\alpha$ is realized by such a $\mathbb{S}ol^3$-manifold
if and only if $\alpha\bar\alpha=1$,
$|\alpha+\bar\alpha|>2$ and $\alpha+\bar\alpha\equiv2$ {\it mod} $(4)$.
Determining the possible ideal classes represented by
$\sqrt\pi$ is more complicated.

\begin{theorem}
Let $\alpha$ be a quadratic unit which is not a root of unity, 
and let $J$ be a nonzero ideal in $\mathbb{Z}[\alpha]$.
Let $A$ be the automorphism of $J\cong\mathbb{Z}^2$ 
given by left multiplication by $\alpha$, 
and let $\kappa=J\rtimes_A\mathbb{Z}$.
Then 

$\kappa$ is a subgroup of index $2$ in a $\mathbb{S}ol^3$-group $\pi$ 
with $\pi/\sqrt\pi\cong{D_\infty}$

$\Leftrightarrow\,\alpha\bar\alpha=1$,
$\alpha\equiv1$ {\it mod} $2\mathbb{Z}[\alpha]$ and there are $\lambda,\mu\not=0\in\mathbb{Z}[\alpha]$ 
and $v,w\in{J}$ such that $\lambda\overline{J}=\mu{J}$,
$\lambda\bar{v}=\mu{v}$ and
$\lambda\bar{w}=\bar\alpha\mu{w}$,
but $\bar\lambda{v}\not=\bar\lambda{j}+\bar\mu\bar{j}$ 
and $\bar\lambda{w}\not=\bar\lambda{j}+\alpha\bar\mu\bar{j}$
for any $j\in{J}$.

Given $\alpha$, the number of isomorphism classes 
of such groups $\pi$  is finite.
\end{theorem}

\begin{proof}
Suppose that $\pi=\langle\kappa,u\rangle$ with $\pi/\sqrt\pi\cong{D_\infty}$
and $[\pi:\kappa]=2$,
and that $t\in\kappa$ generates $\kappa$ {\it mod} $\sqrt\pi$.
Then $t^{-1}$ is conjugate to $t$,
and so $A$ and $A^{-1}$ have the same characteristic polynomial.
Since $trA\not=0$, $\alpha\bar\alpha=\det{A}=1$.

Let $B(j)=uju^{-1}$ and $f(j)=\overline{B(j)}$, for all $j\in{J}$.
Then $B$ is an isomorphism of groups and $f:J\to\overline{J}$ 
is an isomorphism of $\mathbb{Z}[\alpha]$-modules.
Let $v=u^2$ and $w=(tu)^2$.
Then $B^2=(AB)^2=I$, $Bv=v$ and $ABw=w$.
Since $A$ has infinite order, $B\not=I$, and so $\det{B}=-1$.
Moreover, $B\equiv{AB}\equiv{I_2}$ {\it mod} (2),
since $\langle{J,u}\rangle$ and $\langle{J,tu}\rangle\cong\pi_1(Kb)$.
Therefore $A\equiv{I_2}$ {\it mod} (2) also, 
and so $\alpha\equiv1$ {\it mod} $2\mathbb{Z}[\alpha]$.

Since $\pi$ is torsion free,
$(uj)^2$ and $(tuj)^2$ are nontrivial, for all $j\in{J}$.
Equivalently, $v\not\in(I+B)J$ and $w\not\in(I+AB)J$. 

The isomorphism $f$ extends to an automorphism $f_\mathbb{Q}=id_\mathbb{Q}\otimes{f}$ of $\mathbb{Q}[\alpha]$, 
as a vector space over itself.
We may write $f_\mathbb{Q}(1)=\frac\mu\lambda$, 
for some nonzero $\lambda,\mu\in\mathbb{Z}[\alpha]$.
(Note that $\frac{\mu\bar\mu}{\lambda\bar\lambda}=
\det{f_\mathbb{Q}}=-\det{B}=1$.) 
Then $\lambda{f(j)}=\mu{j}$, for all $j\in{J}$,
since $\mathbb{Z}[\alpha]$ is an integral domain.
The linear conditions on $v$ and $w$ become
$\lambda\bar{v}=\mu{v}$ and $\lambda\bar{w}=\bar\alpha\mu{w}$,
while  $\bar\lambda{v}\not=\bar\lambda{j}+\bar\mu\bar{j}$
and $\bar\lambda{w}\not=\bar\lambda{j}+\alpha\bar\mu\bar{j}$
for any $j\in{J}$.

Conversely, suppose that these conditions hold.
Let $Bj=\overline{\frac\mu\lambda{j}}$, for all $j\in{J}$, 
and let $\pi$ be the group with presentation
\[
\langle\kappa,u|u^2=v,~utu^{-1}=t^{-1}wv^{-1},
~uju^{-1}=Bj~\forall{j}\in{J}\rangle.
\]
Then $\pi$ is torsion free and has $\kappa$ as a subgroup of index 2. 
and so is a $\mathbb{S}ol^3$-group.
Clearly $\pi/\sqrt\kappa\cong{D_\infty}$,
and so $\sqrt\kappa\leq\sqrt\pi\leq\kappa$. 
Hence $\sqrt\pi=\sqrt\kappa$ and $\pi/\sqrt\pi\cong{D_\infty}$.

Since $\kappa$ has trivial centre the extensions of $Z/2Z$ by $\kappa$ 
are determined by the image in $Out(\kappa)$ 
of the action of $Z/2Z$ on $\kappa$.
Since There are finitely many groups $\kappa$ realizing $\alpha$, 
by Theorem 1, and $Out(\kappa)$ is finite,
by Theorem 8.10 of \cite{Hi}, there are finitely many
such groups $\pi$.
\end{proof}

In particular, the ideal class $[J]$ must be fixed by the Galois involution.
For example, 
if $\alpha\bar\alpha=1$ and $\alpha\equiv1$ {\it mod} $2\mathbb{Z}[\alpha]$ 
then $J=\mathbb{Z}[\alpha]$, 
$v=1$ and $w=\alpha$ satisfy the other conditions,
with $\lambda=\mu=1$.

Note that if $\alpha$ is a quadratic unit such that $\alpha\bar\alpha=1$
and $\delta=\alpha-1\in2\mathbb{Z}[\alpha]$ then
$\bar\delta=-\alpha^{-1}\delta\in2\mathbb{Z}[\alpha]$ also, and so
$\alpha+\bar\alpha=2-\delta\bar\delta\equiv2$ {\it mod} (4).
(This is equivalent to an earlier matrix argument.)

Every subgroup of finite index in $\pi$ can be generated by three elements,
while proper subgroups of infinite index need at most two generators.
If a nontrivial normal subgroup $N$ has infinite index 
in $\pi$ then it has Hirsch length $\leq2$. 
Hence it is abelian, and so has finite index in $\sqrt\pi$.
Thus proper quotients of a $\mathbb{S}ol^3$-group $\pi$
either have two ends or are finite.

\section{the mod-$2$ cohomology ring}

Martins has constructed an explicit free resolution $P_*\to\mathbb{Z}$
of the augmentation $\mathbb{Z}[\pi]$-module, 
and a diagonal approximation $\Delta:P_*\to{P_*\otimes{P_*}}$, 
which he used to compute the integral and mod-$p$ cohomology rings
for semidirect products 
$\pi\cong\mathbb{Z}^2\rtimes_\Theta\mathbb{Z}$ with $\Theta\in{GL(2,\mathbb{Z})}$ \cite{Ma12}.

We shall take a somewhat different approach, 
first computing cup products into $H^2(\pi;\mathbb{F}_2)$ 
and then using Poincar\'e duality.
Our strategy in determing relations in $H^2(\pi;\mathbb{F}_2)$ 
shall be to use restrictions to subgroups (such as $\sqrt\pi$)
and epimorphisms to quotient groups
(such as $\pi/\sqrt\pi$ or small finite 2-groups),
with known cohomology rings.

We shall usually write $H_*(X)$ and $H^*(X)$ for the homology 
and cohomology of a space or group $X$, 
with coefficients $\mathbb{F}_2$,
and denote the cup-product by juxtaposition.
In each case considered below, 
the given generators for a group $G$ represent a basis for $H_1(G)$, 
and we shall use the corresponding Kronecker dual bases for 
$H^1(G)=Hom(H_1(G),\mathbb{F}_2)$.

\begin{lemma}
Let $w=w_1(\pi)$. Then $w\alpha\beta=\alpha^2\beta+\alpha\beta^2$,
for all $\alpha, \beta\in{H^1}(\pi)$.
In particular, if $w=0$ then $\alpha^2\beta=\alpha\beta^2$ and
$(\alpha+\beta)^3=\alpha^3+\beta^3$.
\end{lemma}

\begin{proof}
The first assertion follows from the Wu relation
$Sq^1z=w\cup{z}$ for all $z\in{H}^{n-1}(X)$,
which holds for any $PD_n$-complex $X$.
The second follows easily.
\end{proof}

If $G$ is a group let $X^n(G)=\langle{g^n}|g\in{G}\rangle$ 
be the subgroup generated by all $n^{th}$ powers.
The next lemma is a refinement of Theorem 2 of \cite{Hi87}
(which is restated here  as part (1) of the lemma).

\begin{lemma}
Let $G$ be a group, and $\rho,\phi,\psi\in{H^1}(G)$.
Let $K=\mathrm{Ker}(\rho)$ and $L=K\cap\mathrm{Ker}(\phi)$.
Then
\begin{enumerate}

\item the kernel of cup product from the symmetric product $\odot^2H^1(G)$
to $H^2(G)$ is the dual of $X^2(G)/X^4(G)[G,X^2(G)]$;

\item the canonical projections  induce isomorphisms

\noindent$H^1(G/X^2(K))\cong{H^1}(G/X^2(L))\cong{H^1}(G/X^4(G))\cong{H^1}(G)$;

\item $\rho\phi=0$ in $H^2(G)\Leftrightarrow 
\rho\phi=0$ in $H^2(G/X^2(K))$;

\item $\phi^2=\rho\phi+\rho\psi$ in $H^2(G)
\Leftrightarrow\phi^2=\rho\phi+\rho\psi$ in $H^2(G/X^2(L))$.
\end{enumerate}
\end{lemma}

\begin{proof}
Part (1) is Theorem 2 of \cite{Hi87}, while part (2) is clear.

If $\phi\psi=0$ in $H^2(G)$ then there is a 1-cochain $F:G\to\mathbb{F}_2$
such that $\phi(g)\psi(h)=\delta{F}(g,h)=F(gh)+F(g)+F(h)$, 
for all $g,h\in{G}$.
Part (3) follows easily, 
since $F$ restricts to a homomorphism on $K$,
and is constant on cosets of $X^2(K)$.

Part (4) is similar.
\end{proof}

In many of the cases considered here,
the coefficients in the linear relations determining 
the kernel of cup product may be found by restricting
to 2-generator subgroups.
However, this is not always enough to determine the triple products
in $H^3(\pi)$.

\begin{lemma}
Let $\{T,Y\}$ be the basis for $H^1(D_8)$ corresponding to the presentation
$D_8=\langle{t,y}|t^2=y^4=1,~tyt^{-1}=y^{-1}\rangle$.
Then $(T+Y)Y=0$ in $H^2(D_8)$.
\end{lemma}

\begin{proof}
Let $D_\infty$ have the presentation $\langle{u,v}|u^2=v^2=1\rangle$,
and let $U,V$ be the dual basis for $H^1(D_\infty)$.
Then $H^*(D_\infty)=\mathbb{F}_2[U,V]/(UV)$.
Let $f:D_\infty\to{D_8}$ be the epimorphism given by $f(u)=t$
and $f(v)=ty$.
Then $f$ induces an isomorphism $D_\infty/X^4(D_\infty)\cong{D_8}$,
so  $H^2(f)$ is injective.
Since $f^*U=T+Y$ and $f^*V=Y$, 
we see that $(T+Y)Y=0$ in $H^2(D_8)$.
\end{proof}

Let $E$ be the ``almost extraspecial" 2-group with presentation
\[
\langle{t,u,v}\mid{t^2=1},~u^2=v^2,~tut^{-1}=u^{-1},~tv=vt,~uv=vu\rangle.
\]

\begin{lemma}
Let $\{T,U,V\}$ be the basis for $H^1(E)$ 
corresponding to the above presentation.
Then $TU+U^2+V^2=0$ in $H^2(E)$.
\end{lemma}

\begin{proof}
Since $X^2(E)\cong{Z/2Z}$, the kernel of cup product from
$\odot^2H^1(G)$ to $H^2(G)$ has dimension 1 \cite{Hi87}.
Thus there is an unique nontrivial linear relation
$aT^2+bU^2+cV^2+dTU +eTV +fUV=0$ in $H^2(E)$.
The coefficients can be determined by restriction to the subgroups
$\langle{t}\rangle\cong{Z/4Z}$,
$\langle{t,u}\rangle\cong{D_8}$,
$\langle{t,v}\rangle\cong{Z/4Z}\oplus{Z/2Z}$,
and $\langle{u,v}\rangle\cong{Z/4Z}\oplus{Z/2Z}$.
\end{proof}

\section{ mapping tori}

Suppose that $\pi\cong\mathbb{Z}^2\rtimes_\Theta\mathbb{Z}$,
where $\Theta=\left(\begin{smallmatrix}
a&c\\ 
b&d
\end{smallmatrix}\right)\in{GL}(2,\mathbb{Z})$.
Let $\varepsilon=ad-bc=\pm1$ and $\tau=a+d$.
Let $\Delta_1=\det(\Theta-I_2)=1-\tau+\varepsilon$ and $\Delta_2=(a-1,b,c,d-1)$
be the elementary divisors of $\Theta-I_2$.
Then $\Delta^2_2$ divides $\Delta_1$, and 
\[
\pi^{ab}\cong\mathbb{Z}\oplus {Z/(\Delta_1/\Delta_2)Z}\oplus{Z/\Delta_2Z}.
\]
Let $\beta=\beta_1(\pi;\mathbb{F}_2)$.
Then $1\leq\beta\leq3$, and $\beta_2(\pi;\mathbb{F}_2)=\beta$,
by Poincar\'e duality.
Let $\rho:\pi\to{Z/2Z}$ be the unique epimorphimorphism
which factors through $\pi/\sqrt\pi\cong\mathbb{Z}$.
If $\pi$ is non-orientable then $\rho=w_1(M)$,
and $K=\pi^+$, the maximal orientable subgroup of $\pi$. 

1. If $\tau$ is odd then $\Delta_1$ is odd and
$\pi^{ab}\cong\mathbb{Z}\oplus{odd}$.
In this case $\rho$ is the unique epimorphism 
from $\pi$ to ${Z/2Z}$, 
and 
\[H^*(\pi)\cong\mathbb{F}_2[\rho,\Xi]/(\rho^2,\Xi^2),\]
where $\Xi$ has degree 2, by Poincar\'e duality.

2. If $\tau\equiv\varepsilon-1$ {\it mod} $(4)$ then 
$\pi^{ab}\cong\mathbb{Z}\oplus{Z/2Z}\oplus{odd}$, and $\beta=2$.
Hence $H^1(\pi)=\langle{\rho,\sigma}\rangle$,
where $\sigma$ does not factor through $Z/4Z$.
Moreover, 
if $G=\pi/X^4(\pi)$ then $X^2(G)\cong(Z/2Z)^2$ is central in $G$.
Thus $\rho^2=\rho\sigma=0$, by Lemma 4, while $\sigma^2\not=0$.
Hence $H^2(\pi)=\langle\sigma^2,\Xi\rangle$,
for some $\Xi$ of degree 2.
Duality then implies that $\sigma^3=\rho\Xi\not=0$.
We may assume also that $\sigma\Xi=0$,
and so
\[
H^*(\pi)\cong\mathbb{F}_2[\rho,\sigma,\Xi]/
(\rho^2,\rho\sigma,\sigma\Xi,\rho\Xi+\sigma^3,\Xi^2).
\]

3. If $\tau\equiv\varepsilon+1$ {\it mod} (4) and $\Delta_2$ is odd
then $\pi^{ab}\cong\mathbb{Z}\oplus{Z/2^kZ}\oplus{odd}$, 
for some $k\geq2$.
Hence $H^1(\pi)=\langle\rho,\sigma\rangle$,
where $\sigma^2=\rho^2=0$. 
Since $\rho\sigma=0$, by the nondegeneracy
of Poincar\'e duality, 
\[
H^*(\pi)\cong\mathbb{F}_2[\rho,\sigma,\Xi,\Omega]/
(\rho^2,\rho\sigma,\sigma^2,\rho\Omega,\sigma\Xi,\rho\Xi+\sigma\Omega,
\Xi^2,\Omega^2,\Xi\Omega),
\]
where $\Xi$ and $\Omega$ have degree 2.

In all the remaining cases $\beta=3$.
For if $\tau\equiv\varepsilon+1$ {\it mod} $(4)$ 
and $\Delta_2$ is even
then $a$ and $d$ are odd and $b$ and $c$ are even.
Hence $\Delta_1=2^kq$ and $\Delta_2=2^\ell{q'}$,
where $0<\ell\leq\frac{k}2$ and $q, q'$ are odd.
In this case $\pi^{ab}\cong\mathbb{Z}\oplus{Z/2^{k-\ell}Z}\oplus{Z/2^\ell{Z}}\oplus{odd}$,
so the images of $\{t,x,y\}$ form a basis for $H_1(\pi)$.
Let $\{\rho,\sigma,\psi\}$ be the dual basis, so that
\[
\sigma(x)=\psi(y)=1\quad\mathrm{and}\quad
\sigma(t)=\sigma(y)=\psi(t)=\psi(x)=0.
\]
If $G=\pi/X^4(\pi)$ then 
$X^2(G)=\langle{t^2,x^2,y^2}\rangle\cong(Z/2Z)^3$ is central in $G$,
so the kernel of cup product from $\odot^2H^1(\pi)$ to $H^2(\pi)$ has rank 3.
It then follows from Poincar\'e duality that $H^*(\pi)$ 
is generated as a ring by $H^1(\pi)$.
In each case,
$\rho\sigma^2=\rho\rho\sigma=0$ and $\rho\psi^2=\rho\rho\psi=0$,
by Lemma 3.
Hence $\rho\sigma\psi\not=0$,
by the nondegeneracy of Poincar\'e duality.
It then follows easily that $\rho\sigma$, $\rho\psi$ and $\sigma\psi$ are
linearly independent,
and so form a basis for $H^2(\pi)$.
We may write
\[
\sigma^2=m\rho\sigma+n\rho\psi+p\sigma\psi\quad
\mathrm{
and}\quad\psi^2=q\rho\sigma+r\rho\psi+s\sigma\psi,
\]
for some $m,\dots,s$.
On restricting to $\sqrt\pi$, we see that $p=s=0$,
since $\sigma^2|_{\sqrt\pi}=\psi^2|_{\sqrt\pi}=0$ and $\rho|_{\sqrt\pi}=0$,
while $\sigma\psi|_{\sqrt\pi}\not=0$.
Since $\rho\sigma^2=\rho^2\sigma=\rho\psi^2=\rho^2\psi=0$,
taking cup products with $\sigma$ and $\psi$ gives
\[
\sigma^3=n\rho\sigma\psi,\quad\sigma^2\psi=m\rho\sigma\psi,
\quad\psi^3=q\rho\sigma\psi\quad\mathrm{and}\quad\sigma\psi^2=r\rho\sigma\psi.
\]

4. If $\ell\geq2$ then $a\equiv{d}\equiv1$ and $b,c\equiv0$ {\it mod} (4), 
so $\varepsilon\equiv1$ {\it mod} (4) also,
i.e., $\pi$ is orientable.
In this case $\sigma^2=\psi^2=\rho^2=0$, and so
\[
H^*(\pi)\cong\mathbb{F}_2[\rho,\sigma,\psi]/(\rho^2,\sigma^2,\psi^2).
\]

Suppose now that $\ell=1$.

5. If $\pi$ is orientable and $\Delta_1\equiv0$ {\it mod} (8) 
we may assume that one of $\sigma$, $\psi$ or $\sigma+\psi$ 
factors through $Z/4Z$.
Thus either $\sigma^2=0$, $\psi^2=0$ or $\sigma^2=\psi^2$.
We may assume that $\sigma^2\not=0$.
Then $\rho\sigma^2=\rho^2\sigma=0$
and $\psi\sigma^2=\psi^2\sigma=0$, and so $\sigma^3\not=0$,
by the nonsingularity of Poincar\'e duality.
Hence
\[
H^*(\pi)\cong\mathbb{F}_2[\rho,\sigma,\psi]/(\rho^2,\rho\psi+\sigma^2,\psi^2).
\]
In this case we see that $\phi^3=0\Leftrightarrow\phi^2=0$.

If $\pi$ is orientable and $\Delta_1\equiv4$ {\it mod} (8) 
then $\tau\equiv6$ {\it mod} (8) and $a,d$ are odd,
and so $a\equiv{d}$ {\it mod} (4).
In this case $\psi^2\not=0$ and $(\sigma+\psi)^2\not=0$ also,
and so $\sigma^2=m\rho\sigma+n\rho\psi$ and 
$\psi^2=q\rho\sigma+r\rho\psi$ are linearly independent.
Hence $mr+nq=1$ in $\mathbb{F}_2$.
Since $w=0$, $\sigma^2\psi=\sigma\psi^2$ and so $m=r$.

6. Suppose first that $a\equiv1$ {\it mod} (4).
Then $bc\equiv4$ {\it mod} (8), and so $b\equiv{c}\equiv2$ {\it mod} (4).
Let $L_\phi=\mathrm{Ker}(\rho)\cap\mathrm{Ker}(\phi)$.
Then $\pi/X^2(L_\phi)$ has a presentation
\[
\langle {t,x,y}\mid t^4=x^4=y^2=1,~tx=xt,~tyt^{-1}=x^2y,~xy=yx\rangle.
\]
Let $J=\langle{t,x}\rangle\cong(Z/4Z)^2$.
Then $\sigma^2|_J=\rho\psi|_J=0$,
while $\rho\sigma|_J\not=0$. 
Applying part (3) of Lemma 4, we see that $m=0$,
and so $\sigma^2=\rho\psi$ and $\psi^2=\rho\sigma$.
(Note, however, that Lemma 4 does {\it not\/} assert that the relation
$\psi^2=q\rho\sigma+r\rho\psi$ also holds in $\pi/X^2(L_\phi)$!
For this, 
we could use $L_\psi=\mathrm{Ker}(\rho)\cap\mathrm{Ker}(\psi)$ instead.)
Hence
\[
H^*(\pi)\cong\mathbb{F}_2[\rho,\sigma,\psi]/(\rho^2,\rho\psi+\sigma^2,
\rho\sigma+\psi^2).
\]
In particular, $\sigma^3=\psi^3=(\rho+\sigma)^3=(\rho+\psi)^3\not=0$.

If $a\equiv-1$ {\it mod} (4) then $bc\equiv0$ {\it mod} (8).
If, say, $b\equiv2$ {\it mod} (4) (so $c\equiv0$ {\it mod} (4))
then the change of basis $x'=x$, $y'=xy$
reduces this case to the one just considered.
In terms of the given basis, we have
\[
H^*(\pi)\cong\mathbb{F}_2[\rho,\sigma,\psi]/(\rho^2,
\rho\sigma+\sigma^2,\rho\psi+\sigma^2+\psi^2).
\]
In this case $\sigma^3\not=0$, but $\psi^3=0$.
A similar result holds if $b\equiv0$ {\it mod} (4) 
and $c\equiv2$ {\it mod} (4).

7. If, however, $a\equiv-1$ {\it mod} (4) and $b\equiv{c}\equiv0$ {\it mod} (4)
then $\pi/X^4(\pi)$ has a presentation
\[
\langle {t,x,y}\mid t^4=x^4=y^4=1,~txt^{-1}=x^{-1},~tyt^{-1}=y^{-1},
~xy=yx\rangle.
\]
In this case $J=\langle{t,x}\rangle$ is non-abelian,
and $\sigma^2|_J\not=0$, while $\rho\psi|_J=0$.
Hence we must have $m=r=1$. 
It is clear from the symmetry of the presentation for
$\pi/X^4(\pi)$ that we must also have $n=q$ in this case,
and so $n=q=0$.
Thus 
\[
H^*(\pi)\cong\mathbb{F}_2[\rho,\sigma,\psi]/(\rho^2,
\rho\sigma+\sigma^2,\rho\psi+\psi^2).
\]
We now find that $\phi^3=0$ for all $\phi\in{H^1}(\pi)$.

If $\ell=1$ and $M$ is non-orientable then 
$a$ and $d$ are odd, and $\Delta_1=-a-d\equiv0$ {\it mod} (4).
In this case $\rho=w_1(M)$, 
and so $\sigma^2\psi+\sigma\psi^2=\rho\sigma\psi\not=0$, by Lemma 3.
After swapping $x$ and $y$, if necessary,
we may assume that $a\equiv1$ {\it mod} (4).

8. If $bc\equiv0$ {\it mod} (8) then, 
after a further change of basis of the form $x'=x, y'=xy$ or $x'=xy, y'=y$, 
if necessary, we may assume that $b\equiv{c}\equiv0$ {\it mod} (4).
Then $\sigma^2=0$, and 
$\pi/\langle\langle{t^2,x,y^4}\rangle\rangle\cong{D_8}$,
so $(\rho+\psi)\psi=0$ also.
Hence
\[
H^*(\pi)\cong\mathbb{F}_2[\rho,\sigma,\psi]/
(\rho^2,\sigma^2,\rho\psi+\psi^2).
\]
In particular, $(\sigma+\psi)^3=(\rho+\sigma+\psi)^3\not=0$,
and all other classes have cube 0.
In terms of the given bases, the other cases are:

If ${b}\equiv0$ and $c\equiv2$ {\it mod} (4)
then 
\[
H^*(\pi)\cong\mathbb{F}_2[\rho,\sigma,\psi]/
(\rho^2,\sigma^2+\psi^2,\rho\psi+\psi^2,\sigma^2\psi).
\]
Here $\sigma^3=(\rho+\sigma)^3\not=0$ 
and all other classes have cube 0.

If ${b}\equiv2$ and ${c}\equiv0$ {\it mod} (4) 
then
\[
H^*(\pi)\cong\mathbb{F}_2[\rho,\sigma,\psi]/(\rho^2,
\sigma^2,\psi^2+\rho\sigma+\rho\psi).
\]
Here $\psi^3=(\rho+\psi)^3\not=0$ 
and all other classes have cube 0.

9. If ${b}\equiv{c}\equiv2$ {\it mod} (4)
then $\sigma^2$ and $\psi^2$ are linearly independent.
There are three distinct epimorphisms from $\pi$ to 
the almost extraspecial group $E$,
given by $f(x)=u^{-1}v), f(y)=u$;
$g(x)=v,g(y)=uv^{-1}$; and $h(x)=v, h(y)=u$.
Using these epimorphisms to pull back the relation given in Lemma 5, 
we find that
\[
H^*(\pi)\cong\mathbb{F}_2[\rho,\sigma,\psi]/(\rho^2,
\sigma^2+\rho\psi,\psi^2+\rho\sigma+\rho\psi).
\]
In particular, every epimorphism $\phi\not=\rho$ has nonzero cube.

\section{unions of twisted $I$-bundles}
Suppose that $\pi/\sqrt\pi\cong{D_\infty}$.
Then $\pi$ is orientable, and has a presentation
\[
\langle
u,v,y\mid uyu^{-1}=y^{-1},~v^2=u^{2a}y^b,~
vu^{2c}y^dv^{-1}=u^{-2c}y^{-d}
\rangle,
\]
where $ad-bc=\pm1$ and $abcd\not=0$.
Let $B=\langle{u,y}\rangle$ and $C=\langle{v,u^{2c}y^d}\rangle$.

If $b$ is odd then $\pi^{ab}\cong{Z/4cZ}\oplus{Z/4Z}$,
where the summands are generated by $u$ and $u^{-a}v$, respectively.
Let $U(u)=V(v)=1$, $U(v)=a$ and $V(u)=0$.
Then
\[
H^*(\pi)\cong\mathbb{F}_2[U,V,\Xi,\Omega]/
(U^2,UV,V^2,U\Xi+V\Omega,\Xi^2,\Omega^2,\Xi\Omega),
\]
where $\Xi$ and $\Omega$ have degree 2.

If $b$ is even then $\pi^{ab}\cong{Z/4cZ}\oplus(Z/2Z)^2$ 
and the images of $u,v$ and $y$ represent a basis for $H_1(\pi)$.
Let $\{U,V,Y\}\in{H^1}(\pi)$ be the dual basis.
Then $U^2$, $V^2$ and $Y^2$ are all nonzero, 
but $W=U+V$ lifts to a homomorphism from $\pi$ to $Z/4Z$, and so $W^2=0$.
Hence $U^2=V^2$.
Since $U$ and $V$ are induced from classes in $H^1(D_\infty)$
we have $UV=0$.
We also have $UY|_B=Y^2|_B$ and $VY|_C=Y^2|_C$, while $U|_C,V|_B,U^2|_B$
and $V^2|_C$ are all 0.

Suppose that $pU^2+qY^2+rUY+sVY=0$ in $H^2(\pi)$.
On restricting to the subgroups $B$ and $C$,
we find that $q+r=q+s=0$.
Since $U^2\not=0$ we must have $q=r=s=1$.
Multiplying by $U$ and $V$, 
we find that $UY^2+U^2Y=0$ and $VY^2+V^2Y=0$.
Poincar\'e duality for $\pi$ now implies that $\{U^2,Y^2,UY\}$
is a basis for $H^2(\pi)$, while $UY^2=U^2Y=VY^2$
generates $H^3(\pi)$.
We see also that $U^3=U^2V=UV^2=V^3=(U+V)^3=0$,
while $(U+Y)^3=(V+Y)^3=(U+V+Y)^3=Y^3$.

Suppose first that $b\equiv0$ {\it mod} (4).
Then $G=\pi/\langle\langle{uv,u^2,y^4}\rangle\rangle\cong{D_8}$. 
Hence $(U+V+Y)Y=0$ in $H^3(\pi)$.
It follows easily that $Y^3=0$, and so all cubes are 0 in $H^3(\pi)$.

If $b\equiv2$ {\it mod} (4) then 
$\pi/\langle\langle{u^2,(uv)^2,v^4,y^4}\rangle\rangle$
has a presentation
\[
\langle{u,v,y}\mid{u^2=(uv)^2=v^4=1},~uyu^{-1}=vyv^{-1}=y^{-1},~v^2=y^2\rangle
\]
Hence there is an epimorphism $f:\pi\to{E}$,
given by $f(u)=t$, $f(v)=u$ and $f(y)=u^{-1}t^{-1}v$.
Since $f^*T=U+Y$, $f^*U=V+Y$, $f^*V=Y$ and $UV=0$,
it follows from Lemma 6 that $UY+VY+V^2+Y^2=0$ in $H^2(\pi)$. 
Multiplying by $Y$, we find that $UY^2+Y^3=0$ and so $Y^3\not=0$.
In this case, only the cubes induced from $H^*(\pi/\sqrt\pi)$ are zero.
 
\section{the borsuk-ulam index}

We may identify an epimorphism $\phi$ with a nonzero
class in $H^1(N;\mathbb{F}_2)$.
Then $BU(N,\phi)=1\Leftrightarrow\phi$ lifts to an integral class $\Phi\in{H^1(N;\mathbb{Z})}$,
while
$BU(N,\phi)=n\Leftrightarrow\phi^n\not=0$ in $H^n(N;\mathbb{F}_2)$
In general, $1\leq{BU(N,\phi)}\leq{n}$.
See \cite{GHZ}.
When $n=3$ the remaining possibility is that
$BU(M,\phi)=2\Leftrightarrow\phi^2=0$ 
but $\phi$ is not the reduction of an integral class.

Suppose first that $\pi/\sqrt\pi\cong\mathbb{Z}$.
Then the following results are immediate from \S3.

1. If $\rho:\pi\to{Z/2Z}$ is the unique epimorphism which factors 
through $\pi/\sqrt\pi\cong\mathbb{Z}$ then $BU(M,\rho)=1$.

2. If $\tau\equiv\varepsilon-1$ {\it mod} $(4)$ then 
$BU(M,\phi)=3$ for all $\phi\not=\rho$.

3. If $\tau\equiv\varepsilon+1$ {\it mod} (4) and either
$\Delta_2$ is odd or
$a\equiv{d}\equiv1$ {\it mod} (4) and $b,c$ are 
divisible by 4, 
then $BU(M,\phi)=2$ for all $\phi\not=\rho$.

4. If $\varepsilon=1$, $\Delta_1\equiv0$ {\it mod} (8) 
and $\Delta_2\equiv2$ {\it mod} (4) then $BU(M,\phi)=2$ 
for the two epimorphisms $\phi\not=\rho$ such that $\phi^2=0$
(i.e, that factor through $Z/4Z$) and
$BU(M,\phi)=3$ for the four such that $\phi^2\not=0$. 

5. If $\varepsilon=1$, $\Delta_1\equiv4$ {\it mod} (8) and
$\Theta\equiv-I_2$ {\it mod} (4)
then $BU(M,\phi)=2$ for all $\phi\not=\rho$.

6. If $\varepsilon=1$ and $\Delta_1\equiv4$ {\it mod} (8),
but $\Theta\not\equiv-I_2$ {\it mod} (4),
then $BU(M,\phi)=2$ for the two epimorphisms 
$\phi\not=\rho$ such that $\phi^2=0$ 
and $BU(M,\phi)=3$ for the four such that $\phi^2\not=0$. 

7. If $\varepsilon=-1$, $\tau\equiv0$ {\it mod} (4),
$\Delta_2\equiv2$ {\it mod} (4) and $bc\equiv0$ {\it mod} (8)
then $BU(M,\phi)=2$ for the four epimorphisms $\phi\not=\rho$ such that $\phi^3=0$ and $BU(M,\phi)=3$ for the two such that $\phi^3\not=0$. 

8. If $\varepsilon=-1$, $\tau\equiv0$ {\it mod} (4),
$\Delta_2\equiv2$ {\it mod} (4) and $bc\equiv4$ {\it mod} (8)
then $BU(M,\phi)=3$ for all $\phi\not=\rho$. 

Suppose now that $\pi/\sqrt\pi\cong{D_\infty}$.
Then the following results are immediate from \S4.

9. If $\pi^{ab}\cong{Z/4cZ}\oplus{Z/4Z}$ then
$BU(M,\phi)=2$ for all $\phi$. 

10. If $\pi^{ab}\cong{Z/4cZ}\oplus(Z/2Z)^2$ and $b\equiv0$ {\it mod} (4)
then $BU(M,\phi)=2$ for all $\phi$.

11. If $\pi^{ab}\cong{Z/4cZ}\oplus(Z/2Z)^2$ and $b\equiv2$ {\it mod} (4)
then $BU(M,\phi)=2$ for epimorphisms $\phi$
which factors through $\pi/\sqrt\pi$,
while $BU(M,\phi)=3$ otherwise.

\section{other geometries}

We remark finally that similar arguments may be used to determine the $\mathbb{F}_2$-cohomology rings and Borsuk-Ulam invariants for pairs 
$(N,\phi)$ with $N$ a closed $\mathbb{E}^3$- or $\mathbb{N}il^3$-manifold.
These manifolds are all Seifert fibred over flat 2-orbifolds. 
Since they have been covered in \cite{BGHZ}, 
we shall confine ourselves to some brief observations.

The ten closed flat 3-manifolds may be easily treated individually. 
The only one admitting a class $\phi$ with $\phi^3\not=0$
has group $G_4$, 
with holonomy $Z/4Z$ and abelianization $\mathbb{Z}\oplus{Z/2Z}$.
Thus $H^1(\pi)=\langle{T,X}\rangle$, where $T^2=0$ and $X^2\not=0$.
We may deduce that $TX=0$ also, by mapping $G_4$ onto $D_8$.
It follows easily that 
\[
H^*(G_4)\cong\mathbb{F}_2[T,X,\Omega]/(T^2,TX,X\Omega,
T\Omega+X^3,\Omega^2),
\]
where $\Omega$ has degree 2.
(Thus $X^3=(T+X)^3\not=0$. 
These classes correspond to the two epimorphisms without integral lifts.)

The possible Seifert bases $B$ of closed $\mathbb{N}il^3$-manifolds 
are the seven flat 2-orbifolds with no reflector curves:
$B=T$, $Kb$, $S(2,2,2,2)$, $S(2,4,4)$, $S(2,3,6)$, $S(3,3,3)$ or $P(2,2)$.
Let $\beta=\pi_1^{orb}(B)$ be the orbifold fundamental group of the base. 
Then $\pi^{ab}$ is an extension of $\beta^{ab}$ by a finite cyclic
group $Z/qZ$, if the base is orientable ($B\not=Kb$ or $P(2,2)$),
and by $Z/(2,q)Z$ otherwise.
The ring $H^*(\pi)$ depends only on the base $B$
and the residue of $q$ {\it mod} (4).

If $B=T$ or $Kb$ then $\pi\cong\mathbb{Z}^2\rtimes_\Theta\mathbb{Z}$,
for some $\Theta\in{GL(2,\mathbb{Z})}$.
These are in fact the cases requiring most effort.
In all other cases $\pi^{ab}$ is finite, 
and the projection of $\pi$ onto $\beta$
induces an isomorphism $H_1(\pi)\cong{H_1(\beta)}$.
When $B=S(2,3,6)$ or $S(3,3,3)$ this group is cyclic.
(In particular, 
such $\mathbb{N}il^3$-manifolds are neither mapping tori
nor unions of twisted $I$-bundles.)
When $B=S(2,4,4)$ we have
$\pi/X^4(\pi)\cong\beta/X^4(\beta)\cong{G_4}/X^4(G_4)$.
The cases of $S(2,2,2,2)$ and $P(2,2)$ are related to those of the flat 3-manifolds $G_2$ and $B_4$, respectively.

The Borsuk-Ulam Theorem and its applications and
extensions are treated in detail in the book \cite{Mat}.


\begin{thebibliography}{99}

\bibitem{GHZ} Gon\c{c}alves, D.L., Hayat, C., Zvengrowski, P.
The Borsuk-Ulam theorem for manifolds, with applications to dimensions two and three, 

in {\it Group actions and homogeneous spaces}, 

Fak. Mat. Fyziky Inform. Univ. Komensk\'eho, Bratislava (2010), 9--28.

\bibitem{BGHZ} Bauval, A., Gon\c{c}alves, D.L., Hayat, C. and Zvengrowski, P. 
The Borsuk-Ulam Theorem for double coverings of Seifert manifolds, 

Proceedings of the Brazilian-Polish Topology meeting, July 2012.

\bibitem{jfc} Carlson, J.F. Cohomology of 2-groups,

www.math.uga.edu/~jfs/groups2/

\bibitem{Hi87} Hillman, J.A. The kernel of integral cup product,

J. Austral. Math. Soc. 43 (1987), 10--15.

\bibitem{Hi} Hillman, J.A. {\sl Four-Manifolds, Geometries and Knots},

GT Monographs 5, Geometry and Topology Publications (2002). 

Latest revision: see\quad http://www.maths.usyd.edu.au/u/jonh/ .

\bibitem{Ma12} Martins, S.T. 
{\it Aproxima\c ciones da diagonal e an\'eis da cohomologia dos grupos
fundamentales das superf\'icies, de fibrados de toros e de certos grupos virtualmente c\'iclicos,}
PhD thesis, Universidade de S\~ao Paulo (2012).

\bibitem{Mat} Matou\v{s}ek, J. (with A.Bj\"orner and G.M.Ziegler)
{\it Using the Borsuk-Ulam Theorem.
Lectures on topological methods in combinatorics and geometry},

Universitext, 
Springer-Verlag, Berlin -- Heidelberg -- New York (2003). 

\bibitem{Ro} Robinson, D.J.S. {\it A Course in the Theory of Groups},

Graduate Texts in Mathematics 80, 

Springer-Verlag, Berlin - Heidelberg - New York (1982).               
 
\end{thebibliography}
\end{document}